\documentclass[12pt]{article}

\usepackage[english]{babel}

\usepackage{amsfonts, amssymb, amsmath, amsthm, epsf}

\usepackage{graphicx}

\usepackage{cite}

\usepackage{dsfont}

\DeclareMathOperator{\Diff}{Diff}

\DeclareMathOperator{\FinHolSh}{FinHolSh}

\newcommand{\Z}{\mathds{Z}}
\newcommand{\R}{\mathds{R}}
\newcommand{\N}{\mathds{N}}
\newcommand{\Leb}{\mbox{Leb}}
\newcommand{\dist}{\mbox{dist}}
\newcommand{\lam}{\lambda}
\newcommand{\ep}{\varepsilon}

\newcommand{\Holder}{\mbox{H\"older }}

\newtheorem{theorem}{Theorem}
\newtheorem{lemma}{Lemma}
\newtheorem{conjecture}{Conjecture}

\theoremstyle{definition}
\newtheorem{remark}{Remark}
\newtheorem{defin}{Definition}

\newcommand{\lr}[1]{\left(#1\right)}

\begin{document}

\title{Shadowing in linear skew products}

\author{Sergey Tikhomirov
\footnote{Max Plank Institute for Mathematics in the Sciences, Inselstrasse 22, Leipzig, 04103, Germany; Chebyshev Laboratory, Saint-Petersburg State Univeristy, 14th line of Vasilievsky Island, 29B, Saint-Petersburg, 199178, Russia; sergey.tikhomirov@gmail.com}
}

\date{}

\maketitle

\begin{abstract}
We consider a linear skew product with the full shift in the base and nonzero Lyapunov exponent in the fiber. We provide a sharp estimate for the precision of shadowing for a typical pseudotrajectory of finite length. This result indicates that the high-dimensional analog of Hammel-Yorke-Grebogi's conjecture \cite{Yorke1, Yorke2} concerning the interval of shadowability for a typical pseudotrajectory is not correct. The main technique is reduction of the shadowing problem to the ruin problem for a simple random walk. Bibliography 22 titles.
\end{abstract}

\section{Introduction}

The theory of shadowing of approximate trajectories
(pseudotrajectories) of dynamical systems is now a well-developed
part of the global theory of dynamical systems (see
the monographs \cite{PilBook, PalmBook} and \cite{PilRev} for a survey of modern results). The shadowing problem is related to the following question: under which conditions, for any pseudotrajectory of $f$ does there exist a close trajectory?

Let us consider a metric space $(G, \dist)$ and  a continuous map $f : G \to G$, $d>0$. For an interval $I = (a, b)$, where $a \in \Z \cup \{-\infty\}$, $b \in \Z \cup \{+\infty\}$, a sequence of points $\{y_k\}_{k \in I}$ is called a \textit{$d$-pseudotrajectory} if the following inequalities hold:
$$
\dist(y_{k+1}, f(y_k)) < d, \quad k \in \Z, \quad k, k+1 \in I.
$$

\begin{defin}\label{defStSh}
We say that $f$ has the \textit{shadowing property}
if for any $\ep > 0$ there exists $d>0$ such that for any
$d$-pseudotrajectory $\{y_k\}_{k\in \Z}$ there exists a trajectory
$\{x_k\}_{k\in \Z}$ such that
\begin{equation}\label{sh}
\dist(x_k, y_k) < \ep, \quad k \in \Z.
\end{equation}
In this case, we say that the pseudotrajectory
$\{y_k\}$ is \textit{$\ep$-shadowed} by~$\{x_k\}$.
\end{defin}

The study of this problem was originated by Anosov \cite{Ano} and Bowen \cite{Bow}. This theory is closely related to the classical theory of structural stability.

Let $G$ be a smooth compact Riemannian manifold of class $C^{\infty}$ without
boundary with  metric $\dist$ and let $f \in \Diff^1(G)$.
It is well known that a diffeomorphism has the shadowing
property in a neighborhood of a hyperbolic set \cite{Ano, Bow} and a
structurally stable diffeomorphism has the shadowing property on the
whole manifold \cite{Rob, Saw}.
At the same time, it is easy to give an example of a diffeomorphism
that is not structurally stable but has shadowing property (see \cite{PilVar}, for instance). Thus, structural stability is not equivalent to shadowing.

Relation between shadowing and structural stability was studied in several contexts. It is known that the $C^1$-interior of the set of diffeomorphisms having the shadowing property coincides with the set of structurally stable diffeomorphisms \cite{Sak} (see \cite{PilRodSak} for a similar result for the orbital shadowing property). Abdenur and Diaz conjectured that a $C^1$-generic
diffeomorphism with the shadowing property is structurally stable;
they have proved this conjecture for the so-called tame diffeomorphisms~\cite{AbdDiaz}.

Analyzing the proofs of the first shadowing results by Anosov~\cite{Ano} and Bowen~\cite{Bow}, it is easy to see that, in a neighborhood of a hyperbolic set, the shadowing property is Lipschitz (and the same holds in the case of a structurally stable diffeomorphism~\cite{PilBook}).
\begin{defin}\label{defLipSh}
We say that $f$ has the \textit{Lipschitz shadowing property} if there exist $\ep_0, L_0 > 0$ such that for any $\ep < \ep_0$ and $d$-pseudotrajectory $\{y_k\}_{k\in \Z}$ with $d = \ep/L_0$ there exists a trajectory $\{x_k\}_{k\in \Z}$ such that inequalities~\eqref{sh} hold.
\end{defin}
Recently \cite{PilTikhLipSh} it was proved that
a diffeomorphism $f \in C^1$ has Lipschitz shadowing property if and only if it is structurally stable
(see~\cite{PilVar, LipPerSh} for a similar results for  periodic and variational shadowing properties).

In the present paper, we are interested which type of shadowing is possible for non-hyperbolic diffeomorphisms.
The following notion will be important for us \cite{TikhHol}:
\begin{defin}\label{defFinHolSh}
We say that $f$ has the \textit{finite \Holder shadowing property} with
exponents \hbox{$\theta \in (0, 1)$}, $\omega \geq 0$ ($\FinHolSh(\theta, \omega)$) if there exist $d_0, L, C > 0$ such that for any $d < d_0$ and
$d$-pseudotrajectory $\{y_k\}_{k \in [0, Cd^{-\omega}]}$ there exists a trajectory
$\{x_k\}_{k \in [0, Cd^{-\omega}]}$ such that
\begin{equation}\label{shHol}\notag
\dist(x_k, y_k) < L d^{\theta}, \quad k \in [0, Cd^{-\omega}].
\end{equation}
\end{defin}

S. Hammel, J. Yorke, and C. Grebogi made the following conjecture  based on results of numerical experiments \cite{Yorke1, Yorke2}:
\begin{conjecture}\label{ConjHGY}
A typical dissipative map $f: \R^2  \to \R^2$ with positive Lyapunov exponent satisfies $\FinHolSh(1/2, 1/2)$.
\end{conjecture}
In \cite{Yorke1, Yorke2}, the precise mathematical meaning of word ``typical'' was not provided.

There are plenty of not structurally stable examples satisfying $\FinHolSh(1/2, 1/2)$, for instance \cite[Example 1]{TikhHol} and the identity map.

In the present paper, we study this conjecture for a model example: a linear skew product (see the definition in Section \ref{secMain}). We give lower and upper bounds for the precision of shadowing of finite length pseudotrajectories.
These bounds show that, depending on parameters of the skew product diffeomorphism, it might satisfy and not satisfy analog of Conjecture \ref{ConjHGY}.

We expect that similarly to works \cite{Gor1, Gor2}, such a skew product can be embedded into a diffeomorphism of a manifold of dimension  4. This would allow us to construct an open set of diffeomorphisms violating a high-dimensional analog of Conjecture \ref{ConjHGY}. Similarly, we can construct an open set of diffeomorphisms satisfying this conjecture. However, we did not implement such a construction  and leave it out of the scope of the present paper.

Note that in \cite{TikhHol} it was shown that Conjecture \ref{ConjHGY} cannot be improved (see also \cite{PujKor} for the discussion on \Holder shadowing for 1-dimensional maps):
\begin{theorem}\label{thmFinHolSh}
If a diffeomorphism $f \in C^2$ satisfies $\FinHolSh(\theta, \omega)$ with
\begin{equation}\notag
\theta > 1/2, \quad \theta + \omega > 1,
\end{equation}
then $f$ is structurally stable.
\end{theorem}

The paper is organized as follows. In Section \ref{secMain}, we formulate exact statements of the results. In Section \ref{secRW}, we  formulate a particular problem for random walks and prove its equivalence to the shadowing property. In Section \ref{secProof}, we give a proof of the main result.

\section{Main Result}\label{secMain}
Let $\Sigma = \{0, 1\}^{\Z}$. Endow $\Sigma$ with the standard probability measure $\nu$ and the following  metric:
$$
\dist(\{\omega^i\}, \{\tilde{\omega}^i\}) = 1/2^k, \quad \mbox{where $k = \min\{|i|: \omega^i \ne \tilde{\omega}^i\}$}.
$$
For a sequence $\omega = \{ \omega^i \} \in \Sigma$ denote by $t(\omega)$ the 0th element of the sequence: $t(\omega) = \omega^0$. Define the ``shift map'' $\sigma: \Sigma \to \Sigma$ as follows:
$$
\lr{\sigma(\omega)}^i = \omega^{i+1}.
$$

Consider the space $Q = \Sigma \times \R$. Endow $Q$ with the product measure $\mu = \nu \times \Leb$ and the maximum metric:
$$
\dist((\omega, x), (\tilde{\omega}, \tilde{x})) = \max(\dist(\omega,\tilde{\omega}), \dist(x, \tilde{x})).
$$
For $q \in Q$ and $a>0$ denote by $B(a, q)$ the open ball of radius $a$ centered at~$q$.

Fix $\lam_0, \lam_1 \in \R$ satisfying the following conditions
\begin{equation}\label{eqlam}
0 < \lam_0 < 1 < \lam_1, \quad \lam_0 \lam_1 \ne 1.
\end{equation}
Consider the map $f: Q \to Q$ defined as follows:
$$
f(\omega, x) = (\sigma(\omega), \lam_{t(\omega)}x).
$$

For $q \in Q$, $d > 0$, $N \in \N$ let $\Omega_{q, d, N}$ be the set of $d$-pseudotrajectories of length $N$ starting at $q_0 = q$. If we consider $q_{k+1}$ being chosen at random in $B(d, f(q_k))$ uniformly with respect to the measure $\mu$, then $\Omega_{q, d, N}$ forms a finite time Markov chain. This naturally endows $\Omega_{q, d, N}$ with a probability measure $P$. See also \cite{YorkeAbsolute} for a similar concept for infinite pseudotrajectories.

For $\ep > 0$ let $p(q, d, N, \ep)$ be the probability of pseudotrajectory in $\Omega_{q, d, N}$ to be $\ep$-shadowable.
Note that this event is measurable since it forms an open subset of $\Omega_{q, d, N}$.

\begin{lemma}\label{lemq}
Let $q = (\omega, x)$, $\tilde{q} = (\omega, 0)$. For any $d, \ep > 0$, $N \in \N$,  the following equality holds:
$$
p(q, d, N, \ep) = p(\tilde{q}, d, N, \ep).
$$
\end{lemma}

\begin{proof}
Consider $\{q_k = (\omega_k, x_k)\} \in \Omega_{q, d, N}$. Put $r_k := x_{k+1} - \lambda_{t(\omega_k)}x_{k}$. Consider a sequence $\{\tilde{q}_k = (\omega_k, \tilde{x}_k)\}$, where
$$
\tilde{x}_0 = 0, \quad \tilde{x}_{k+1} = \lambda_{t(w_k)}x_k + r_k.
$$
The following holds:
\begin{enumerate}
\item the correspondence $\{q_k\} \leftrightarrow \{\tilde{q}_k\}$ is one-to-one and preserves the probability measure;
\item for any $\ep > 0$ pseudotrajectory $\{q_k\}$ is $\ep$-shadowed by a trajectory of a point $(\omega, x)$ if and only if $\{\tilde{q}_k\}$ is $\ep$-shadowed by a trajectory of a point $(\omega, x-x_0)$.
\end{enumerate}
These statements complete the proof of the lemma.
\end{proof}

For $d, \ep > 0$, $N \in \N$ define
$$
p(d, N, \ep) := \int_{\omega \in \Sigma} p((\omega, 0), d, N, \ep) \mbox{d$\nu$}.
$$
Note that the integral exists since for fixed $d$, $N$, $\ep$, the value $p((\omega, 0), d, N, \ep)$ depends only on a finite number of entries of $\omega$.
The quantity $p(d, N, \ep)$ can be interpreted as the probability of a $d$-pseudotrajectory of length $N$ to be $\ep$-shadowed.


The main result of the paper is the following:
\begin{theorem}\label{thmMain}
For any $\lam_0, \lam_1 \in \R$ satisfying (\ref{eqlam}) there exist $\ep_0 > 0$, $0< c_0 < \infty$ such that for any $\ep < \ep_0$,  the following holds:
\begin{enumerate}
\item If $c < c_0$, then $\lim_{N \to \infty} p(\ep/N^c, N, \ep) = 0$;
\item if $c > c_0$, then $\lim_{N \to \infty} p(\ep/N^c, N, \ep) = 1$.
\end{enumerate}
\end{theorem}

\begin{remark}\label{remlin}
Later (Lemma \ref{lemrw}) we prove that for any $N \in \N$, $L > 0$,  $\ep_1, \ep_2 \in (0, \ep_0)$, the equality $p(\ep_1/L, N, \ep_1) = p(\ep_2/L, N, \ep_2)$ holds.
Hence the result of Theorem \ref{thmMain} actually does not depend on the value of $\ep$.
\end{remark}

\begin{remark}\label{remY}
Due to Remark \ref{remlin} analog of the Hammel-Grebogi-Yorke conjecture for map $f$  suggests that $p(\ep/N, N, \ep)$ is close to 1. Hence, if $c_0 > 1$, then Hammel-Grebogi-Yorke conjecture is not satisfied. For an example of such parameters see Remark \ref{remex}.
\end{remark}

\section{Equivalent Formulation}\label{secRW}

Let $a_0 = \ln \lam_0$, $a_1 = \ln \lam_1$.
Consider the following random variable:
$$
\gamma =
\begin{cases}
    a_0 \quad \mbox{with probability 1/2},\\
    a_1 \quad \mbox{with probability 1/2}.
\end{cases}
$$

Fix $N>0$. Consider the random walk $\{A_i\}_{i \in [0, \infty)}$ generated by $\gamma$ and independent uniformly distributed in $[-1, 1]$ variables $\{r_i\}_{i \in [0, \infty)}$.
Define a sequence $\{z_i\}_{i \in [0, N]}$ as follows:
\begin{equation}\label{eqz}
z_0 = 0, \quad z_{i+1} = z_i + \frac{r_{i+1}}{e^{A_{i+1}}}.
\end{equation}
For given sequences $(\{A_i\}_{i \in [0, N]}, \{r_i\}_{i \in [0, N]})$ define
$$
B(k, n) := \frac{e^{A_k + A_n}}{e^{A_k} + e^{A_n}}|z_n - z_k| = \frac{e^{A_n}}{e^{A_k} + e^{A_n}}\left|e^{A_k}z_n - e^{A_k}z_k\right|,
$$
$$
K(\{A_i\}, \{r_i\}) := \max_{0 \leq k < n \leq N} B(k, n),
$$
$$
s(N, L) := P(K(\{A_i\}_{i \in [0, N]}, \{r_i\}_{i \in [0, N]})< L),
$$
where $P(\cdot)$ is the probability of a certain event.

Below we prove the following lemma.
\begin{lemma}\label{lemrw}
There exist $\ep_0 > 0$, $L_0> 0$ such that for any $d \geq 0$, $L > L_0$, $N \in \N$ satisfying $Ld < \ep_0$ the following equality holds:
$$
p(d, N, Ld) = s(N, L).
$$
\end{lemma}

\begin{proof}
Let us choose $\ep_0, L_0 > 0$ such that if $\dist(\omega, \tilde{\omega}) < \ep_0$, then $t(\omega) = t(\tilde{\omega})$ and the  map $\sigma$ satisfies the Lipschitz shadowing property with constants $\ep_0, L_0$.

Fix $d < d_0$, $N > 0$ and $L > L_0$ satisfying $Ld < \ep_0$.
Let us choose $\omega$ at random according to the probability measure $\nu$ and a pseudotajectory $\{q_k\} = \{(\omega_k, x_k)\} \in \Omega_{(\omega, 0), d, N}$ according to the measure $P$ (see Section~ \ref{secMain}).
Consider the sequences
$$
\gamma_k = a_{t(\omega_k)}, \quad A_k = \sum_{i = 0}^{k}\gamma_i, \quad r_k = (x_{k} - \lambda_{t(\omega_{k-1})}x_{k-1})/d.
$$
Note that $r_k$ are independent uniformly distributed in $[-1, 1]$ and $\gamma_k$ are independent and distributed according to $\gamma$.

Below we prove that the sequence $\{q_k\}$ can be $Ld$-shadowed if and only if
\begin{equation}
\label{eqL}
L \geq K(\{A_i\}, \{r_i\}).
\end{equation}

Assume that the pseudotrajectory $(\omega_k, x_k)$ is $Ld$-shadowed by an exact trajectory $(\xi_k, y_k)$.
By the choice of $\ep_0$, the following equality holds:
\begin{equation}
\label{eqxi}
t(\omega_k) = t(\xi_k).
\end{equation}

Now let us study the behavior of the second coordinate.
Note that
\begin{equation}\label{eqy}
y_{k+1} = \lambda_{t(\xi_k)}y_k = e^{\gamma_k}y_k, \quad
y_n = e^{A_n - A_k}y_k,
\end{equation}
$$
x_n = e^{A_n - A_k}x_k + e^{A_k}(z_n - z_k),
$$
where $z_k$ are defined by \eqref{eqz}. Hence,
$$
(y_n - x_n) = e^{A_n - A_k}(y_k - x_k) + e^{A_k}(z_n - z_k).
$$
From this equality it is easy to deduce that
$$
\max(|y_k-x_k|, |y_n - x_n|) \geq B(k, n)
$$
and the equality holds if $(y_k - x_k) = -(y_n-x_n)$. Hence, inequality \eqref{eqL} holds.

\medskip

Now let us assume that \eqref{eqL} holds and prove that $(w_k, x_k)$ can be $Ld$-shadowed. Let us choose a sequence $\{ \xi_k \}$ which $Ld$-shadows $\{w_k\}$, then  equalities \eqref{eqxi} hold.

For $y_0 \in \R$ define $y_k$ by relations \eqref{eqy}
and consider function $F: \R \to \R$ defined as follows:
$$
F(y_0) = \max_{0\leq k \leq N}|y_k - x_k|.
$$
Since the function $F$ is continuous, it is easy to show that it attains a minimum for some $y_0$.
Denote $L':= \min_{y_0 \in \R} F(y_0)$ and let $y_0$ be such that $L' = F(y_0)$. Let $D = \{ k \in [0, N]: |y_k - x_k| = F(y_0)\}$.
Let us consider two cases.

\textit{Case 1.} For all $k \in D$ the value $y_k - x_k$ has the same sign. Without loss of generality, we can assume that these values are positive. Then for small enough $\delta > 0$, the inequality $F(y_0 - \delta) < F(y_0)$ holds, which contradicts the choice of $y_0$.

\textit{Case 2.} There exists indices $k, n \in D$ such that the values
$y_k - x_k$ and $y_n - x_n$ have different signs. Then $(y_k - x_k) = -(y_n-x_n)$, and hence $L' = B(k, n) \leq K(\{A_i\}, \{z_i\})$.

\end{proof}

\section{Proof of Theorem \ref{thmMain}}\label{secProof}

Note that shadowing problems for the maps $f$ and $f^{-1}$ are equivalent (up to a constant multiplier at $d$). In what follows, we assume that $\lam_0 \lam_1 > 1$. Put
$$
v := E(\gamma) = (a_0+a_1)/2 > 0, \quad M := (\ln N)^2, \quad w := v/2.
$$

In the proof of Theorem \ref{thmMain}, we use the following statements.
\begin{lemma}[{Large Deviation Principle, \cite[Secion 3]{Var}}]\label{LDP}
There exists an increasing function $h: (0, \infty) \to (0, \infty)$ such that for any $\ep > 0$ and $\delta > 0$ and for large enough $n$, the following inequalities hold:
\begin{equation}\notag
P\lr{\frac{A_n}{n} - E(\gamma) < -\ep} < e^{-(h(\ep)-\delta)n}.
\end{equation}
\begin{equation}\notag
P\lr{\frac{A_n}{n} - E(\gamma) < -\ep} > e^{-(h(\ep)+\delta)n}.
\end{equation}
\end{lemma}

\begin{lemma}[Ruin Problem, {\cite[Chapter XII, \S 4, 5]{Feller2}}]\label{lemC}
Let $b$ be the unique positive root of the equation
$$
\frac{1}{2}\lr{e^{-b a_0} + e^{-b a_1}} = 1.
$$
For any $\delta > 0$ and for large enough $C > 0$, the following inequalities hold:
\begin{equation}
\label{eqb1}
P(\exists i \geq 0 : A_i \leq -C) \leq e^{-C(b-\delta)},
\end{equation}
\begin{equation}
\label{eqb2}
P(\exists i \geq 0 : A_i \leq -C) \geq e^{-C(b+\delta)},
\end{equation}
\end{lemma}

Put $c_0 = 1/b$.
Due to Lemma \ref{lemrw}, it is enough to prove the following:
\begin{itemize}
\item[(S1)] If $c < c_0$, then $\lim_{N \to \infty} s(N, N^c) = 0$.
\item[(S2)] If $c > c_0$, then $\lim_{N \to \infty} s(N, N^c) = 1$.
\end{itemize}

\begin{remark}\label{remex}
For $\lam_0 = 1/2$, $\lam_1 = 3$ the inequalities $b< 1$, $c_0 > 1$ hold, and hence by Remark~\ref{remY} the statement of Conjecture \ref{ConjHGY} does not hold.
Similarly, $c_0 > 1$ for $\lam_0 = 1/3$, $\lam_1 = 2$.
\end{remark}

Below we prove items (S1) and (S2).

\subsection{Proof of (S1)}

Assume that $c < 1/b$. Let us choose $c_1 \in (c, 1/b)$ and $\delta > 0$ satisfying
\begin{equation}\label{eq1}
c_1(b+\delta) < 1.
\end{equation}

Consider the following events:
\begin{align*}
I & =
\left\{
\exists i \in [0, M]: A_i \leq -c_1 \ln N; \; \mbox{and} \; A_{2M} \geq 0
\right\},\\
I_1 & = \left\{ \exists i \in [0, M]: A_i \leq -c_1 \ln N \right\},\\
I_2 & = \left\{ \exists i \in [0, M]: A_i \leq -w M \right\},\\
I_3 & = \left\{ A_{2M} - A_M \leq w M \right\}.
\end{align*}
The following holds:
\begin{equation}
\label{2.1}
P(I) \geq P(I_1) - P(I_2) - P(I_3),
\end{equation}
\begin{align}\notag
P(I_1) & \geq P(\exists i \geq 0: A_i \leq -c_1 \ln N) - P(\exists i > M: A_i \leq -c_1 \ln N) \\
\notag & \geq e^{-c_1 \ln N (b+\delta)} - \sum_{i = M+1}^N P(A_i \leq 0)
 \geq N^{-c_1(b+\delta)} - \sum_{i = M+1}^N e^{-i h(v)} \\
\label{eqI1} & \geq N^{-c_1(b+\delta)} - \frac{1}{1 - e^{-h(v)}} e^{-(M+1) h(v)}  \geq N^{-c_1(b+\delta)} + o(N^{-2}).
\end{align}
Similarly
\begin{equation}\label{eqI2}
P(I_2) \leq \sum_{i = M+1}^{\infty} P(A_i \leq 0) = o(N^{-2}),
\end{equation}
\begin{equation}\label{eqI3}
P(I_3) \leq e^{-Mh(v-w)} = o(N^{-2}).
\end{equation}

From inequalities \eqref{2.1}-\eqref{eqI3} we conclude that
\begin{equation}
\label{eqI}
P(I) \geq N^{-c_1(b+\delta)} + o(N^{-2}).
\end{equation}

Assume that the event $I$ has happened and let $i \in [0, M]$ be one of the indices satisfying the inequality $A_i < -c_1 \ln N$. Note that the following events are independent:
$$
J_1 =\{r_{i} \in [1/2; 1]\}, \quad J_2 = \left\{z_{2M} - z_0 \geq \frac{r_{i}}{e^{A_{i}}}\right\}.
$$
Hence,
$$
P\lr{z_{2M}- z_0 \geq \frac{1}{2e^{A_i}}} \geq P(J_1)P(J_2) = 1/4 \cdot 1/2 = 1/8
$$
and
$$
P(B(0, 2M) > N^{c_1}/4) \geq \frac{1}{8}P(I) = \frac{1}{8}N^{-c_1(b+\delta)} + o(N^{-2}).
$$
Note that for large enough $N$, the inequality $N^c < N^{c_1}/4$ holds, and hence
$$
P(B(0, 2M) > N^c) \geq \frac{1}{8}N^{-c_1(b+\delta)} + o(N^{-2}).
$$
Similarly, for any $k \in [0, N-2M]$,
$$
P(B(k, k+2M) > N^c) \geq \frac{1}{8}N^{-c_1(b+\delta)} + o(N^{-2}).
$$
Note that the events in the last expression for $k = 0, 2M, 2\cdot 2M, \dots ([N/(2M)]-1)2M$ are independent, and hence
\begin{multline}\label{Add6.1}
P\lr{\exists k \in [0, N-2M]: B(k, k+2M) > N^c} \geq \\ 1 - \lr{1- \lr{\frac{1}{8}N^{-c_1(b+\delta)} + o(N^{-2})}}^{[N/(2M)]}.
\end{multline}
Using \eqref{eq1}, we conclude that
\begin{multline*}
\lr{\frac{1}{8}N^{-c_1(b+\delta)} + o(N^{-2})}[N/(2M)] \geq \lr{\frac{1}{8}N^{-c_1(b+\delta)} + o(N^{-2})}\lr{\frac{N}{2(\ln N)^2} - 1} \\
 = \frac{1}{16(\ln N)^2}N^{1- c_1(b+\delta)} + o(N^{-1}) 
 \xrightarrow[N \to \infty]{} \infty
\end{multline*}
and hence
\begin{equation}\label{Add6.2}
\lr{1- \lr{\frac{1}{8}N^{-c_1(b+\delta)} + o(N^{-2})}}^{[N/(2M)]} \xrightarrow[N \to \infty]{} 0.
\end{equation}
Relations \eqref{Add6.1}, \eqref{Add6.2} imply that
$$
P(K(\{A_i\}_{i \in [0, N]}, \{r_i\}_{i \in [0, N]}) > N^c) \xrightarrow[N \to \infty]{} 1.
$$
Hence,
$$
\lim_{N \to \infty} s(N, N^c) = 0.
$$

\subsection{Proof of (S2)}

Let $c > 1/b$. Let us choose $c_1 \in (1/b, c)$ and $\delta > 0$ satisfying $c_1(b-\delta) > 1$.

Note that for any $n > k$ the following inequalities hold:
$$
e^{A_k}|z_n - z_k| \leq \sum_{i = k}^n e^{-(A_i - A_k)},
$$
$$
\frac{e^{A_n}}{e^{A_k} + e^{A_n}} \leq 1.
$$
Hence,
\begin{equation}
\label{Add7.1}
K(\{A_i\}, \{r_i\}) \leq \max_{0\leq k < n \leq N} \sum_{i = k}^n e^{-(A_i - A_k)} \leq \max_{0 \leq k \leq N} \sum_{i = k}^N e^{-(A_i - A_k)} =: D(\{A_i\}).
\end{equation}

The following holds:
\begin{align*}
P(D(\{A_i\}) < N^c) & \geq 1 - P\lr{\exists k \in [0, N] : \sum_{i = k}^N e^{-(A_i - A_k)} > N^c}\\
 & \geq 1 - N P\lr{\sum_{i = 0}^N e^{-(A_i - A_k)} > N^c}.
\end{align*}

Note that if $\sum_{i = 0}^N e^{-(A_i - A_k)} > N^c$, then one of the following inequalities holds:
$$
\exists i \in [0, M]: e^{-A_i} > \frac{N^c}{2M},
$$
$$
\exists i \in [M, N]: e^{-A_i} > \frac{N^{c-1}}{2}.
$$
Note that for large enough $N$, the following inequalities hold:
$$
\frac{N^c}{2M} > N^{c_1}, \quad N^{c - 1}/2 > e^{-w M},
$$
and hence (arguing similarly to the previous section), for large enough $N$,
\begin{align*}
P\lr{\sum_{i = 0}^N e^{-(A_i - A_k)} > N^{c_1}}
& \leq P(\exists i \in [0, M]: A_i < -c_1 \ln N ) + P(\exists i \in [M, N]: A_i < w M) \\
& \leq e^{-(b-\delta)c_1 \ln N} + o(N^{-2}) = N^{-(b-\delta)c_1} + o(N^{-2}).
\end{align*}
Finally,
$$
P(D(\{A_i\}) \leq N^c) \geq 1- N(N^{-(b-\delta)c_1} + o(N^{-2})) \xrightarrow[N \to \infty]{} 1,
$$
and hence relations \eqref{Add7.1} imply that
$$
\lim_{N \to \infty} s(N, N^c) = 1.
$$

\section{Acknowledgment}

The author is grateful to Jairo Bochi for the help in stating the problem and to Artem Sapozhnikov and Yinshan Chang for the help with basics in probability. The author is also grateful to the hospitality of Pontificia Universidade Catolica do Rio de Janeiro where the work was initiated. The work of the author was partially supported by
by Chebyshev Laboratory (Department of Mathematics and Mechanics, St. Petersburg State University)  under RF Government grant 11.G34.31.0026, JSC ``Gazprom neft'', by the Saint-Petersburg State University research grant 6.38.223.2014, and by  the German-Russian Interdisciplinary Science Center (G-RISC) funded by the German Federal Foreign Office via the German Academic Exchange Service (DAAD).

\end{document}